\documentclass[a4paper,oneside,10pt]{amsart}
\usepackage{amsmath,amsfonts,amssymb,amsthm,graphicx}
\usepackage[colorlinks=true,citecolor=blue]{hyperref}

\newcommand{\R}{\mathbb R}
\newcommand{\F}{\mathcal F}
\newcommand{\G}{\mathcal G}
\newcommand{\h}{\mathcal H}
\newcommand{\ES}{\operatorname{ES}}
\newcommand{\ESl}{\ES_l}

\newtheorem{theorem}{Theorem}[section]
\newtheorem{lemma}[theorem]{Lemma}
\newtheorem{prop}[theorem]{Proposition}
\theoremstyle{remark}
\newtheorem*{remark}{Remark}

\title{Erd\H os - Szekeres Theorem for Lines}
\author{Imre B\'ar\'any}
\author{Edgardo Rold\'an-Pensado}
\author{G\'eza T\'oth}


\begin{document}

\begin{abstract}
According to the Erd\H os-Szekeres theorem, for every $n$, a sufficiently
large set of points in general position in
the plane contains $n$ in convex position.
In this note we investigate the line version of this result, that is, we want
to find $n$ lines in convex position in a sufficiently large set of lines that are in general position.
We prove almost matching upper and lower bounds for the minimum size of the
set of lines in general position
that always contains $n$ in convex position. This is quite
unexpected, since in the case of points, the best known bounds are very far
from each other. 
We also establish the dual versions of many variants and generalizations of
the Erd\H os-Szekeres theorem.
\end{abstract}

\maketitle

\section{Introduction}

The classical Erd\H os-Szekeres Theorem \cite{ES1935} states the following:
For every $n\ge 3$ there is a smallest number $\ES(n)$ such that any family of
at least $\ES(n)$ points in general position in $\R^2$ contains $n$ points
which are the vertices of a convex $n$-gon.
There are hundreds of variants and generalizations of this result.
The best known bounds for $\ES(n)$ are
\begin{equation}\label{eq:ESZ}
2^{n-2}+1\le \ES(n)\le \binom{2n-5}{n-2}+1,
\end{equation}
see \cite{ES1960}, \cite{TV2005}. Since $\binom{2n-5}{n-2}+1\approx
4^n/\sqrt{n}$,
there is a huge gap between these bounds.

We consider the following line or ``dual" version of this theorem: For every
$n\ge 2$ there is a smallest number $\ESl(n)$ such that any family of (at
least) $\ESl(n)$ lines in general position in $\R^2$ contains $n$ lines which
determine
the sides of an $n$-cell.
To our surprise we could find almost no trace of this theorem in the literature.
The only exception is the result of Harborth and M\"oller \cite{HM1994}
which states that for every $n$, every pseudoline arrangement with
sufficiently many pseudolines contains a subarrangement of $n$ pseudolines
with a face that is bounded by all of them.
The existence of $\ESl(n)$ follows from this.

Some clarification is needed. A finite family of lines is in general position
if no three lines are concurrent, no two are parallel and no line is
vertical. From now on we assume that every family of lines that we work with
is
in general position.

Let $n\ge k\ge 2$.
A family of $n$ lines (in general position, as we just agreed) in the plane
\emph{defines a $k$-cell} $P$ if $P$ is a connected component of the
complement of the union of these lines, and the boundary of $P$ contains a
segment
from exactly $k$ of the lines.
A $k$-cell is always convex, it is either bounded, and then it is a convex
$k$-gon, or unbounded, and then its boundary contains exactly
two half-lines.
A family of $n$ lines is in \emph{convex position} if it defines an $n$-cell.
A family of lines \emph{spans} an $n$-cell if it contains $n$ lines in convex position.

Note that a family of $3$ lines always defines three $2$-cells and four
$3$-cells out of which exactly one is
a triangle.
A family of $4$ points in the plane may not determine a convex $4$-gon but $4$
lines always define two $4$-cells, a bounded one and an unbounded
one.

Here is a quick proof of the existence of $\ESl(n)$.
Assume $\F=\{a_1,\dots,a_N\}$ is a family of lines where the equation of $a_i$
is $y=m_ix+c_i$ and the ordering is chosen so that $m_1<m_2<\dots<m_N$. The
triple of lines $a_i,a_j,a_k$ with $i<j<k$ is coloured red if the intersection
of $a_i$ and $a_k$ lies below $a_j$, and is coloured blue otherwise. Ramsey's
theorem implies that, for large enough $N$, $\F$ contains $n$ lines so that
all of their triples are of the same colour. It is easy to see that these $n$
lines are in convex position.
The
$n$-cell defined by them is actually unbounded. Note that, in general, the
existence of a bounded $n$-cell cannot be guaranteed. This is shown by the
lines containing $n=(N-1)/2\ge 5$ consecutive edges of a regular $N$-gon
(where $N$ is odd). Here the lines define an unbounded $n$-cell. It is easy to
see, or follows from
Proposition \ref{prop:atmost1}, that all other cells spanned by the lines have
size at most four.

The above argument gives a doubly exponential bound for $N$. A single exponential
bound follows from \eqref{eq:ESZ} and from the fact that, in the projective
plane, the point and line versions of the Erd\H os-Szekeres problem are dual
to each other.
On the other hand, in the affine plane the point and line versions of the
Erd\H os-Szekeres theorem are not dual to each other: the dual of the convex
hull of $n$ points is not a cell (or an $n$-cell) in the arrangement of lines
dual to the points. However something can be saved from duality. Namely, caps
and cups go to cups and caps by duality and this will
be discussed in detail in Section~\ref{sec:upper}.
By this observation we obtain $\ES_l(n)\le \ES(2n)$.

Our main result establishes almost matching upper and lower bounds for $\ES_l(n)$.

\begin{theorem}\label{mainresult}
For $n\ge 3$ we have
$$2\binom{n-4}{\lceil n/2\rceil -1}^2\le \ES_l(n)\le \binom{2n-4}{n-2}.$$
\end{theorem}

Here the lower bound is of order $4^n/n$ and the upper one is of order
$4^n/\sqrt n$. So the lower bound is much larger than the one
in \eqref{eq:ESZ},
which is conjectured to be the true value of $\ES(n)$. Thus one may wonder if this
conjecture, the so called Happy End Conjecture, is correct or not.
In later sections we study variants of the Erd\H os-Szekeres problem for lines,
most of them are inspired by variants of the original one. Further
information on the Erd\H os-Szekeres Theorem can be found in the surveys
\cite{MS2000} and \cite{BK2001}.

\section{Initial values of \texorpdfstring{$\ESl(n)$}{ESl(n)}}

First we compare the known values of $\ES(n)$ and $\ESl(n)$ for small values
of $n$ in Table~\ref{table:values}. Note that any two lines define a $2$-cell
so $\ESl(2)=2$ but $\ES(2)$
does not make much sense.

\begin{table}
\centering
\begin{tabular}{|c|c|c|}
\hline
$n$ & $\ES(n)$ & $\ESl(n)$ \\
\hline
$2$ & - & $2$ \\
$3$ & $3$ & $3$ \\
$4$ & $5$ & $4$ \\
$5$ & $9$ & $7$ \\
$6$ & $17$ & $\ge 15$ \\
\hline
\end{tabular}
\vspace{10pt}
\caption{Known values of $\ES(n)$ and $\ESl(n)$.}
\label{table:values}
\end{table}

The exact value of $\ES(6)$ was confirmed with the help of a computer \cite{SP2006}.
We determined the value of $\ESl(5)$ by a computer analysis of all
possible configurations of $7$ lines. To do this, we used the database on
Tobias Christ's home web-page
\url{http://www.inf.ethz.ch/personal/christt/line_arrangements.php} containing
all possible simple line arrangements of up to $9$ lines \cite[Section
3.2.5]{Chr2011}.
The bound for $\ESl(6)$ is shown in Figure~\ref{fig:ESl_6}.

In principle, a family of $n$ lines may define several $n$-cells, but this
only happens for $n=2,3,4$.

\begin{prop}\label{prop:atmost1}
A family of $n\ge 5$ lines in the plane defines at most one $n$-cell.
\end{prop}

\begin{proof}
Assume it defines two $n$-cells, $P$ and $Q$, say. There is a line in the
family that separates $P$ and $Q$ as otherwise they coincide. There are at
most $4$ distinct lines that are tangent to both $P$ and $Q$. But every line
in the family is tangent to both $P$ and $Q$, so $n \le 4$.
\end{proof}

\section{Upper bounds for \texorpdfstring{$\ESl(n)$}{ESl(n)}} \label{sec:upper}

The original proof \cite{ES1935} of the upper bound in \eqref{eq:ESZ} uses cups and caps.
A set of points $\{p_1,\dots, p_n\}$, ordered by their $x$-coordinates, forms an \emph{$n$-cup} if the point
$p_i$ is below the line $p_{i-1}p_{i+1}$ for every $1<i<n$. If instead, $p_i$
is above the line $p_{i-1}p_{i+1}$ for every $i$,
then this set forms an \emph{$n$-cap}.

For the line case we can do something similar.
Let $\F=\{a_1,\dots,a_n\}$ be a family of $n$ lines in general position.
We say that $\F$ forms an \emph{$n$-cup} (resp. \emph{$n$-cap})
if $\F$ defines an $n$-cell $C$
with the property that the intersection of $C$ with vertical line is a
half-line bounded from below (resp. above). We also call cell $C$
an \emph{$n$-cup} or \emph{$n$-cap}, respectively.

Clearly, any non-vertical line forms a $1$-cup and a $1$-cap at the same
time,
and any two non-vertical, intersecting lines form a $2$-cup and a $2$-cap at the same
time.
Suppose that $\F=\{a_1,\dots,a_n\}$,
$n\ge 3$ and
the equation for line $a_i$ is $y = m_i x + c_i$, $m_1<m_2<\dots<m_n$.
It is easy to see that $\F$ forms an \emph{$n$-cup} if $a_i$ is above $a_{i-1}\cap
a_{i+1}$ for every $1<i<n$. This is also equivalent to $\F$ defining an
unbounded $n$-cell open from above.
Similar observations hold for \emph{$n$-caps}.

With these definitions, the proof by Erd\H os and Szekeres works in the line
case almost without
modification.

As mentioned before, there is a duality such that cups and caps for lines go
to caps and cups for points.
This is achieved by using the map
\[
\{y=mx+c\}\mapsto(m,c),
\]
which is a bijection between the set of non-vertical lines and the set of
points in the plane. Assume we are given three non-parallel
lines ordered according to
slope with
equations
$y=m_ix+c_i$ for $i=1,2,3$. These lines are concurrent if and only if the
corresponding points $(m_i,c_i)$ are collinear.
This is because both conditions are equivalent to the equation
\[
\frac{c_2-c_1}{m_2-m_1}=\frac{c_3-c_2}{m_3-m_2}.
\]
Note that here $m_i=m_j$ is excluded by the general position condition. If the lines form a cup, then this equation will be an inequality and the
left-hand side will be larger. The new inequality is equivalent to the points
forming a cap. If instead the lines form a cap, then the points must form a
cup.

Let $f_l=f_l(k,l)$ be the smallest number of lines needed so that every family
of $f_l$ lines contains either a $k$-cup or an $l$-cap. Applying this duality
to what was shown in \cite{ES1935}, we immediately obtain the following.

\begin{lemma}\label{lem:cupcap}
\[
f_l(k,l)=\binom{k+l-4}{k-2}+1.
\]
\end{lemma}

As a consequence we obtain
\[
\ESl(n)\le f_l(n,n)=\binom{2n-4}{n-2}+1.
\]
This is the same as the bound proved by Erd\H os and Szekeres for $\ES(n)$. In
the case of points however, this bound has been improved upon several times
\cite{CG1998,KP1998,TV1998,TV2005} to finally obtain
\[
\ES(n)\le\binom{2n-5}{n-2}+1.
\]
The techniques used to make these improvements do not readily translate to the
line case, but based on the idea of Kleitman and Pachter \cite{KP1998}, we can
get the smallest possible improvement of
$1$.

\begin{proof}[Proof of Theorem \ref{mainresult}, upper bound]
Let $\F=\{a_1,\dots,a_N\}$ be a family of lines in general position ordered
according to slope.
Let $v$ be the intersection of the lines of smallest and
largest slopes in $\F$. Suppose that $v$ has the largest $y$-coordinate of
all intersections of lines of $\F$. If this condition holds, then $\F$ is
called a \emph{vertical configuration}.

Assume that $\F$ does not contain $n$ lines in convex position. Consider the set of all
intersections of the lines, and take their convex hull $C$. Let $v$ be a
vertex of $C$, given by the intersection of lines $a$ and $a'$.
Take a line $\ell$ that avoids $C$ but $v$ is very close to it.
Any projective transformation which maps $\ell$ to the line at infinity
has the property, that it does not change convexity of the lines since
$\ell$ avoids $C$. By a suitable such projective transformation
we can make $\F$ is a vertical configuration.
So, in the sequel we assume that $\F$ is a vertical configuration and the lines $a$ and
$a'$ correspond to $a_1$ and $a_N$.

Let $A\subset \F$ be the set of lines that form the last line, (that is, the
line of largest slope), of some $(n-1)$-cup in $\F$, and let $B\subset \F$ be
the set of lines that form the first line (the line of smallest slope),
of some $(n-1)$-cap in $\F$.

Suppose that $A\cap B\neq \emptyset$. Then there is an $(n-1)$-cup
$a_{i_1},a_{i_2}, \ldots , a_{i_{n-1}}$ in increasing order, and an
$(n-1)$-cap $a_{j_1}, a_{j_2}, \ldots , a_{j_{n-1}}$, also in increasing order
such that
$a_{i_{n-1}}=a_{j_1}$.
Then, just like in the original Erd\H os-Szekeres argument, either $a_{i_1},
a_{i_2}, \ldots , a_{i_{n-1}} a_{j_2}$ forms an $n$-cup, or
$a_{i_{n-2}},a_{j_1}, a_{j_2}, \ldots , a_{j_{n-1}}$,
forms an $n$-cap, a contradiction.

Therefore, $A$ and $B$ are disjoint. Observe, that $a_1$ can not be the last
line of an $(n-1)$-cup, so it is not in $A$. If $a_1$ is in $B$, then, using
the fact that we have a vertical configuration, the $(n-1)$-cup starting with
$a_1$ can be completed to an $n$-cell with the line $a_N$ (see
Figure~\ref{fig:cupextend}). As this is impossible, $a_1$ is not in $B$
either. Let $A'=\F\setminus B$ and $B'=\F\setminus A$.
Thus $N \le \lvert A'\rvert+\lvert B'\rvert-1$.

\begin{figure}[ht]
\includegraphics{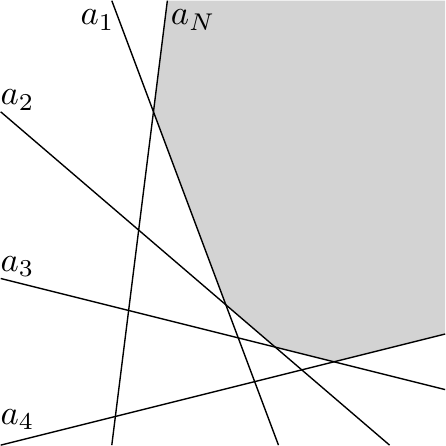}
\caption{A vertical configuration. The lines $a_1,a_2,a_3,a_4$ form a cup,
together with $a_N$ they are still in convex position.}
\label{fig:cupextend}
\end{figure}

If $\lvert A'\rvert\ge \binom{2n-5}{n-2}+1$, then by Lemma \ref{lem:cupcap},
it contains an $n$-cup or an $(n-1)$-cap. Since $n$-cups are excluded, we
have an $(n-1)$-cap, but then we have a line in both $A'$ and $B$, which is
impossible. So $\lvert A'\rvert\le \binom{2n-5}{n-2}$. We can show similarly
that $\lvert B'\rvert\le \binom{2n-5}{n-2}$. Therefore, $N \le\lvert A'\rvert+\lvert B'\rvert-1\le
2\binom{2n-5}{n-2}-1=\binom{2n-4}{n-2}-1$, so $\ESl(n)\le\binom{2n-4}{n-2}$.
\end{proof}

\section{Lower bounds for \texorpdfstring{$\ESl(n)$}{ESl(n)}}

An exponential lower bound follows from
Lemma~\ref{lem:cupcap}. A family with
$\binom{n-4}{\lfloor n/2\rfloor-2}$ lines that does not contain
$\lfloor n/2\rfloor$-caps or $\lceil n/2\rceil$-cups can not contain an $n$-cell. This
gives the bound $\ESl(n)=\Omega(2^n/\sqrt n)$.

We will construct an example which gives the better lower bound of Theorem~\ref{mainresult}.
Some preparations are needed.

If the lines $a_1,\dots,a_n$ form a cap or a cup, then they define an
unbounded $n$-cell $P$. Clearly, any vertical line intersects $P$. An $n$-cell
$P$ defined by lines $a_1,\dots,a_n$ is
\emph{unbounded to the right (resp. left)} if it is unbounded
and it is to the right (resp. left) of some
vertical line. The two lines $a_i,a_j$ for which $P\cap a_i$ and $P\cap a_j$
are unbounded are called
the \emph{end lines of $P$}.

Let $\F$ be a family of lines and let
\[
A=\begin{pmatrix}
\delta & 0 \\
\delta t & \delta^2
\end{pmatrix},\text{ with $t\in\R$ and $\delta>0$}.
\]
Apply an affine transformation of the plane $\vec{v}\mapsto A\vec{v}+\vec{b}$,
where $\vec{v}, \vec{b}\in\R^2$.
This transformation is the composition of a ``flattening''
$(x,y)\mapsto (x,\delta y)$, a scaling
$(x,\delta y)\mapsto(\delta x,\delta^2y)$, a vertical shear transformation
$(\delta x,\delta^2 y)\mapsto (\delta x,\delta tx+\delta^2 y)$ and finally, a
translation
$(\delta x,\delta tx+\delta^2 y)\mapsto (\delta x +b_x,\delta tx+\delta^2 y+b_y)$.
From this, it is not hard to see that a line with slope $m$ goes to a line
with slope $t+\delta m$, hence the ordering of lines according to slope is
preserved. Vertical lines are mapped to vertical lines and their
upward direction is also preserved, so cups, caps, and cells
unbounded to the right and left
remain invariant.
We call these affine transformations
\emph{unbounded-cell-preserving affine transformations}.

Let $a$ be a non-vertical line, $\varepsilon>0$, and let $\F$ be a family of
lines. Let $\F(a, \varepsilon)$ be a family of lines satisfying
the following conditions:

\begin{itemize}
\item[(i)] the slopes of all lines in $\F(a, \varepsilon)$
are within $\varepsilon$ of the slope of $a$;
\item[(ii)] all intersections of the lines
of $\F(a, \varepsilon)$ are below the $x$-axis;
\item[(iii)] the distance between any two intersections
of the lines of $\F(a, \varepsilon)$ is at
most
$\varepsilon$.
\end{itemize}
It is easy to find a suitable unbounded-cell-preserving affine transformation
such that the image of $\F$ satisfies properties
(i), (ii) and (iii).
These new families will be used in the following construction. In
each step we replace each line $a$ with a suitable family $\F(a, \varepsilon)$ that was constructed previously.

\begin{lemma}
Let $k$ and $l$ be positive integers. There is a family $\F_{k,l}$ consisting
of $\binom{k+l-2}{k-1}$ lines that 
spans no $(k+1)$-cap, no $(l+1)$-cap and
no $4$-cell unbounded
to the right.
\end{lemma}
\begin{proof}
The construction is essentially the same as the one from
Lemma~\ref{lem:cupcap}, obtained by dualising the construction of Erd\H os and
Szekeres for points. But here we have to be more careful.

For every positive integer $k$, let $\F_{k,1}=\F_{1,k}=\{a\}$, where $a$ is any horizontal line.
None of these families contain a $2$-cap, $2$-cup or $4$-cell unbounded to the right.

Now suppose that we have families $\F_{k-1,l}$ and $\F_{k,l-1}$ satisfying the
properties.
Let $a_1$ and $a_2$ be two lines, both of positive slopes such that $a_2$ has
greater slope than $a_1$. Assume that they intersect above
the $x$-axis.
Let $\varepsilon$ be a very small number and let $\F_{k,l}=\F_{k-1,l}(a_1,
\varepsilon)\cup\F_{k,l-1}(a_2,\varepsilon)$. See Figure~\ref{fig:capcuplower}.

\begin{figure}[ht]
\includegraphics{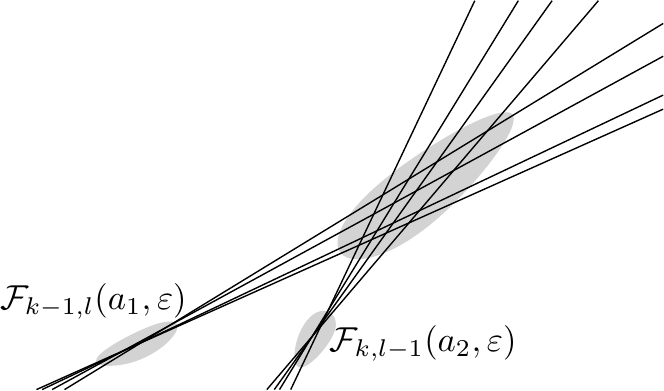}
\caption{The family $\F_{k,l}$, all the line intersections lie within the shaded regions.}
\label{fig:capcuplower}
\end{figure}

Assume that $\G\subset\F_{k,l}$ defines a $(k+1)$-cup, then $\G$ must contain
lines from both $\F_{k-1,l}(a_1,\varepsilon)$ and
$\F_{k,l-1}(a_2,\varepsilon)$. But because $a_1$ has smaller slope than $a_2$,
$\G$ can not contain more than
one line from $\F_{k,l-1}(a_2,\varepsilon)$.
Since the lines in $\G\cap\F_{k-1,l}(a_1,\varepsilon)$ form a cap, it contains
at most $k-1$ lines.
This is a contradiction.
We can show similarly that $\F_{k,l}$ contains no $(l+1)$-cap.

Now assume that $\G\subset\F_{k,l}$ defines a $4$-cell $P$ unbounded to
the right. Let $b_1$ and $b_2$ be the end lines of $P$. It is impossible for
$b_1,b_2\in\F_{k-1,l}(a_1,\varepsilon)$ or
$b_1,b_2\in\F_{k,l-1}(a_2,\varepsilon)$,
therefore we may
assume that $b_1\in\F_{k-1,l}(a_1,\varepsilon)$ and
$b_2\in\F_{k,l-1}(a_2,\varepsilon)$. But then $\G$ can define at most a
$2$-cell unbounded to the right.
\end{proof}

Observe, that if we reflect $\F_{k,l}$ over a vertical line, we obtain a
family of lines $\F'_{k,l}$ with no $(k+1)$-cup, no $(l+1)$-cap, and no
$4$-cell unbounded to the left.

\begin{prop}\label{prop:lower}
For any $n$, there is a family with
$\Theta(4^n/n)$ lines that does not contain $n$ lines in convex position.
\end{prop}

\begin{proof}
Assume for simplicity that
$n=2k+2$ for some $k\ge 4$ and let $\F_{k,k}'=\{a_1,\dots,a_N\}$, where the
lines $a_i$ are ordered according to slope and $N=\binom{2k-2}{k-1}$.
Applying a suitable unbounded-cell-preserving affine
transformation, we can assume that
every $a_i$ has positive slope and all intersections of the lines are above
the $x$-axis.
Now define $\F=\bigcup_i\F_{k,k}(a_i,\varepsilon)$ for some $\varepsilon>0$
very small. The number of elements in $\F$ is
\[
\binom{2k-2}{k-1}^2=\Theta\left(\frac{4^n}{n}\right).
\]

We now show that for $n>4$, $\F$ spans no $n$-cell. Assume for a contradiction
that $\G=\{b_1,\dots,b_n\}$ defines an $n$-cell $C$, $\G\subset\F$.
For any subset $\G'\subseteq\G$ of size $n'\le n$, let $C(\G')$ be the cell
$C'$ defined by $\G'$ which contains $C$.

For each $a_i$, consider the set $\G_i=\F_{k,k}(a_i,\varepsilon)\cap\G$.
Observe, that for any $i$, if $C(\G_i)$ is bounded or
unbounded to the left, then $\G=\G_i$, since no line in $\G\setminus\G_i$
could intersect $C(\G_i)$.
If $C(\G_i)$ is unbounded to the right, then $\G\setminus\G_i$ can contain at
most one line. Therefore, we can assume that for each $i$,
$C(\G_i)$ is a cup or a cap.

We divide the proof into 4 cases depending on the cardinality
of $I=\{i:\lvert\G_i\rvert >1\}$.
The elements of $I$ are just the subscripts $i$ for which
$\F_{k,k}(a_i, \varepsilon)\cap\G$
contains at least $2$ lines.

\begin{figure}[ht]
\centering
\includegraphics{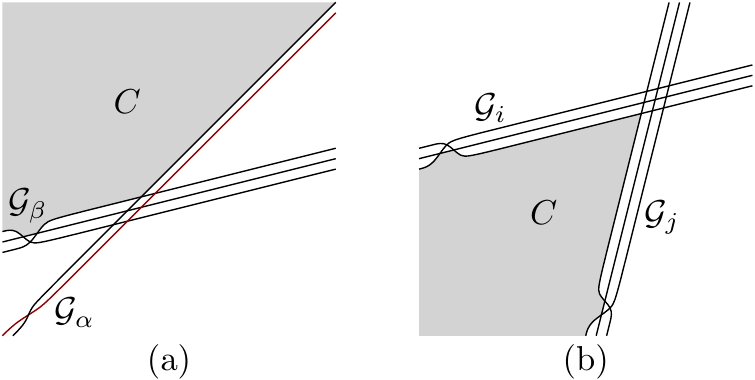}
\caption{Pictures for cases 1 and 2 of Proposition \ref{prop:lower}.}
\label{fig:cases}
\end{figure}

\noindent\emph{Case 1.} Suppose that $\lvert I\rvert\ge 3$.
By the previous observations, for each $i\in I$, $C(\G_i)$ is a cup or a cap.
Assume that $C(\G_{\alpha})$ and $C(\G_{\beta})$ are both cups, $\alpha,\beta \in I$ and $\alpha >
\beta$. Then the line in $\G_{\alpha}$ of the smallest slope can not intersect
$C$,
which is a contradiction (see part (a) of Figure \ref{fig:cases}). The case of two caps can be arranged the same way.

\noindent\emph{Case 2.} If $I=\{i,j\}$ with $i<j$, then $\G$ must define an
$n$-cell unbounded to the left.
Furthermore, $\G_i$ defines a cap and
$\G_j$ defines a cup. This implies that $\G_i\cup\G_j$ contains at most $2k$
lines. Finally, there can be at most one other line in $\G$ because
$a_i$, $a_j$ and the lines in $\G\setminus(\G_i\cup\G_j)$ form a cell
unbounded to the left (see part (b) of Figure \ref{fig:cases}). Thus $\lvert\G\rvert\le 2k+1$, a contradiction.

\noindent\emph{Case 3.} Suppose now that
$I$ contains a single element, $i$.
Let $a'_i$ be an arbitrary element of
$\G_i$ and let $\G'=(\G\setminus\G_i)\cup\{a'_i\}$.
That is, remove all lines of $\G_i$ from $\G$, except one.
The cell $C(\G')$ is unbounded to the left, a cup, or a cap, otherwise
the other lines in $\G_i$ could not intersect it.
Since $\G_i$ is a cup or cap, $\G$ contains at most $k+2$ lines if
$C(\G')$ is unbounded to the left, and at most $2k-1$ lines
if $C(\G')$ is a cup or cap.

%
%
%
%

\noindent\emph{Case 4.} If $I=\emptyset$, then $\G$ is essentially the same as
a subset of $\F_{k,k}'$. It therefore has at
most $2k$ elements.

If $n=2k+1$, then a bound of
$\left(\binom{2k-2}{k-1}+1\right)\binom{2k-3}{k-1}$ can be obtained in the
same way considering
$\F=\F_{k,k}(a_1,\varepsilon)\cup\bigcup_{i\ge 2}\F_{k-1,k}(a_i,\varepsilon)$.
\end{proof}

The lower bound in Proposition~\ref{prop:lower} is weaker than that in Theorem~\ref{mainresult}
by a factor of $2$. The better construction is a direct consequence of the
following result. The proof is similar to that of
Proposition~\ref{prop:lower}, so we only present the construction.
The rest of the proof is left to the reader.

\begin{theorem}\label{th:lower}
Let $k\ge 2$ be an integer. If $n=2k+2$, there is a family
of $$2\binom{2k-2}{k-1}^2$$ lines that spans no $n$-cell.
If $n=2k+1$, there is a family
of $$2\left(\binom{2k-2}{k-1}+1\right)\binom{2k-3}{k-1}$$ lines that
spans no $n$-cell.
\end{theorem}

\begin{proof}
First assume that $n=2k+2$ and let $N=\binom{2k-2}{k-1}$.

Take a line $a$ with slope $1$ and a line $b$ slope $-1$ intersecting above
the $x$-axis. Construct the families
$\F_{k,k}(a, \varepsilon)$ and $\F'_{k,k}(b, \varepsilon)$,
for some very small $\varepsilon$, each with $N$ lines.
Reflect both families about the $x$-axis
and apply a vertical transformation
so that all intersections are above the $x$-axis.
Let $\F$ denote the
resulting family of
$2N$ lines.

Let $\F=\{a_1,\dots,a_{2N}\}$ with the lines
ordered according to slope as usual.
The first $N$ lines are a reflected copy of $\F_{k,k}$ and the last $N$ lines
are a reflected copy of $\F'_{k,k}$.
The resulting configuration is shown in Figure \ref{fig:lower}.

\begin{figure}[ht]
\centering
\includegraphics{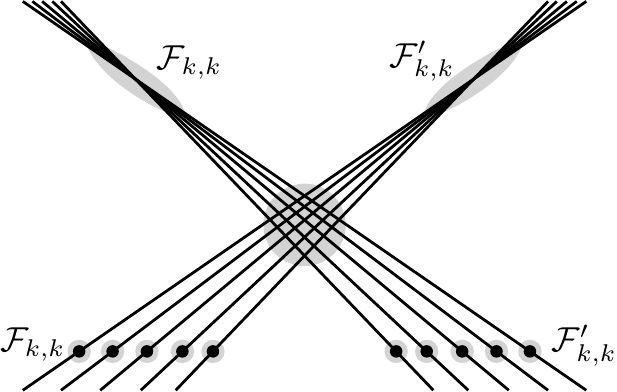}
\caption{Example for the lower bound, all intersections are in the shaded regions.}
\label{fig:lower}
\end{figure}

Finally, we consider $\F'_{k,k}(a_i, \varepsilon)$ for $i\le N$ and
$\F_{k,k}(a_i,\varepsilon)$ for $N<i\le 2N$. The union of these families
contains $2N^2$ lines. An analysis similar to the one made in Proposition
\ref{prop:lower}
shows that it spans no $n$-cell.

The case $n=2k+1$ can be done using the union of the families
$\F'_{k,k}(a_1)\cup\bigcup_{i\ge 2}^N\F'_{k-1,k}(a_i)$ and
$\F_{k,k}(a_{N+1})\cup\bigcup_{i\ge N+2}^{2N}\F_{k-1,k}(a_i)$.
\end{proof}


These examples do not give optimal values for $\ESl(n)$. For small $n$,
similar ideas can be used to give better examples. Figure~\ref{fig:ESl_6}
shows a construction which implies
$\ESl(6)\ge 15$.

\begin{figure}[ht]
\centering
\includegraphics{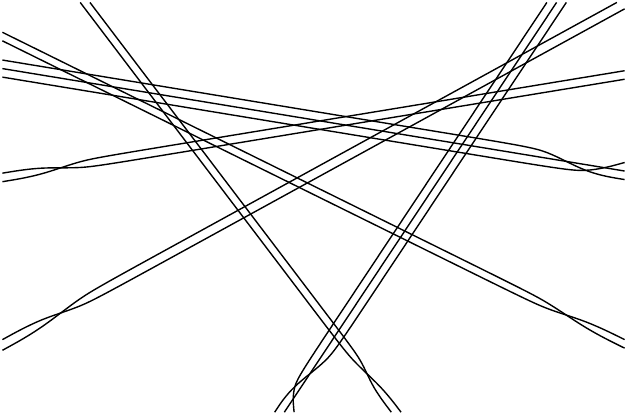}
\caption{$\ESl(6)\ge 15$.}
\label{fig:ESl_6}
\end{figure}

\section{Empty \texorpdfstring{$n$}{n}-cells}

Erd\H os \cite{Erd1981} also asked the following question: Given $n\ge 3$, is
there an integer $N$ such that every family $X$ of $N$ points in general
position contains $n$ points whose convex hull is an empty $n$-gon? An $n$-gon
is \emph{empty} if no point of $X$ lies in its interior. The answer is
positive for $n\le 6$ \cite{Har1978, Nic2007, Ge2008} and negative for $n\ge 7$
\cite{Hor1983}. The same question can be asked with lines instead of
points. We want to find a subfamily $\G$ of a family $\F$ of lines that
defines an $n$-cell such that no line of $\F$ intersects the interior of this
$n$-cell. Such a cell is called an \emph{empty} $n$-cell.

So the question is whether an arrangement $\F$ of sufficiently many lines (in
general position) defines an $n$-cell. Here is an example showing that the
answer is no for
$n\ge 5$.
For every $k\ge 0$ let $x_{2k}=(k,0)$, and for every $k\ge 1$ let $x_{2k-1}=(0,-k)$.
Construct lines $a_i$ inductively with $x_i\in a_i$ for every $i$.
First define $a_0$ and $a_1$ as any two lines through $x_0$ resp. $x_1$ so that $a_0\cap a_1$ is above the $x$-axis.
For each $i>1$, $a_i$ is the line through $x_i$ so that all the intersections
between the lines $a_0,\dots,a_{i-1}$ are above $a_i$ and all points
$x_0,\dots,x_{i-1}$ are below
$a_i$.
The construction up to $a_5$ is shown on Figure \ref{fig:noempty}.
It is easy to see that with each new line we create only bounded and unbounded
cells of at most four sides.

\begin{figure}[ht]
\centering
\includegraphics{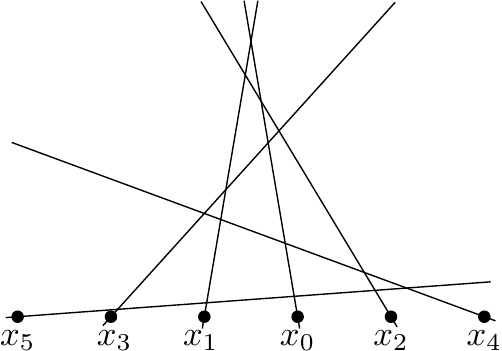}
\caption{A family of lines without empty $n$-cells for $n>4$.}
\label{fig:noempty}
\end{figure}

There are other examples with no empty $n$-gon for $n\ge 5$.
Namely, there are arrangements of pseudo-lines in the Euclidean plane for
which all bounded cells are $3$-cells and $4$-cells. These can be found in
\cite{LLM+2007}, where it is shown that all such examples are
stretchable.

So what remains is to decide whether there is a family of $N$ lines in general
position without empty $3$- or $4$-cells. There is always an empty $3$-cell:
just choose a line $\ell$ from the family $\F$ and let $a,b \in \F$ be the two
lines whose intersection is closest to $\ell$. It is easy to see that $\ell$,
$a$ and $b$ define an empty triangle of $\F$. This argument gives $N/3$
empty triangles
(when $N=|\F|$).

The case of empty $4$-cells is wide open. We could only find one non-trivial
example with no empty $4$-cell.
This is shown in Figure
\ref{fig:no4cell}.
It is worth noting that there is no other example with this
property in the range $3<N<10$.
This was confirmed by a computer.

\begin{figure}[ht]
\centering
\includegraphics{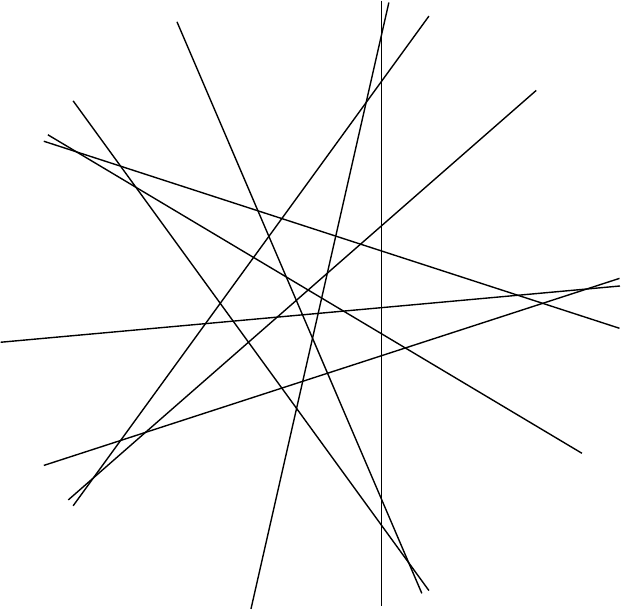}
\caption{An example of $10$ lines with no $4$-cells.}
\label{fig:no4cell}
\end{figure}

The same question has been asked in the projective plane $\R P^2$. It is not
known whether in $\R P^2$ there is family of $N$ lines (in general position)
with no empty $4$-cell for large $N$. The only known non-trivial values for
which there is such a family are $N=6,10,16$
(see \cite{Gru1972}).

It is known \cite{Le1926} that in $\R P^2$ every family of $N\ge 3$ lines (in
general position) defines at least $N$ and at most $N^2/3+O(N)$ empty
$3$-cells. Every line arrangement in the Euclidean plane comes from a line
arrangement in $\R P^2$ by choosing the line at infinity properly, and a
$3$-cell in the projective plane gives rise to a $3$-cell in the Euclidean
plane. It follows that any simple arrangement of $N\ge 3$ lines in the
Euclidean plane determines at least $N$ empty $3$-cells. This bound is sharp,
as is shown by taking the lines forming the sides of a regular
$N$-gon.
For the upper bound we can argue in the following way. Given a line
arrangement of $N$ lines in $\R^2$, we can add to it the line at infinity to
obtain $N+1$ lines in $\R P^2$. The new line may create or destroy at most
$2N$ $3$-cells, therefore there could be at most $N^2/3+O(N)$ empty triangles
in the original arrangement in $\R^2$. Some further examples, results and
questions in this direction can be found in \cite{FP1984} and in
\cite{FK1999}.

In the coloured versions, every two-coloured point set in general position
contains an empty monochromatic triangle, but it is unknown whether there is
always an empty convex monochromatic
quadrilateral.
The case of lines is simpler. Order the lines in the example in Figure
\ref{fig:noempty} according to their intersection of the $x$-axis and colour
them alternately red and blue. Lines $a_0,a_1,\dots,a_n$
define no
empty monochromatic $3$-cell or $4$-cell.
For odd $n$ there is no empty monochromatic $2$-cell either.

\section{The positive fraction version}

In this section we extend the positive fraction Erd\H os-Szekeres theorem from
\cite{BV1998} to the line case. Given families of lines $\F_1,\dots,\F_n$, a
\emph{transversal}
is a set $\{ a_1, \dots , a_n \}$ where $a_i\in\F_i$ for every $i$.

\begin{theorem}\label{th:posfr}
For every $n$ there exists a number $C_n$ such that any family $\F$ of lines
with large enough $\lvert\F\rvert$ contains subfamilies $\F_1,\dots,\F_n$,
each of size at least $C_n\lvert\F\rvert$, such that every transversal of
$\F_1,\dots,\F_n$ defines an $n$-cell.
\end{theorem}

\begin{proof}
The proof is based on the method from \cite{PS1998} as refined in
\cite{PV2002}.
Assume
that $\F$ is a family of $N$ lines in general position and $N> 2m^4$
where
$m=\ESl(4n+1)$. Every $m$-tuple in $\F$ contains a $(4n+1)$-tuple that
defines a $(4n+1)$-cell. The usual double counting argument shows that
the number of $(4n+1)$-cells defined by $(4n+1)$-tuples in $\F$ is at least
\[
\binom{N}{4n+1}\binom{m}{4n+1}^{-1}.
\]

Assume now that $\h=\{a_0,a_1,\dots,a_{4n}\} \subset \F$
defines a
$(4n+1)$-cell $P$ with sides $S_0,S_1,\dots,S_{4n}$ (here $S_i \subset a_i$
for all $i$) where the sides are listed in the order they appear on the
boundary of $P$. Here $S_0$ and $S_{4n}$ are the two half-lines if $P$ is
unbounded, otherwise the common point of $S_0$ and $S_{4n}$ is the point of
$P$ with the greatest $y$-coordinate.
We say in this case that
$\{a_0,a_2,a_4,\dots,a_{4n}\}$ {\sl supports} $\h$. It follows that $\F$ has a
$(2n+1)$-element subset, say $\h^*=\{b_0,b_2,\dots,b_{4n}\}$
that supports at least
\[
\binom{N}{4n+1}\binom{m}{4n+1}^{-1}\binom{N}{2n+1}^{-1} >
\frac {(N-4n-1)^{2n}}{m^{4n}}
\]
$(4n+1)$-tuples in $\F$ that define a $(4n+1)$-cell. The above inequality
holds if $n$ is large enough.
The lines in $\h^*$
define a $(2n+1)$-cell with sides $T_0,T_2,T_4,\dots,T_{4n}$. For
$j=1,3,5,\dots,4n-1$ let $\F_j$ denote the set of lines in $\F$
that intersect
$T_{j-1}$ and $T_{j+1}$. We omit the simple proof of the following fact:
Every transversal of $\F_1,\F_3,\dots,\F_{4n-1}$ defines a $2n$-cell.

Let $t$ be the $n$-th smallest number among the $2n$ integers
$\lvert\F_1\rvert,\lvert\F_3\rvert,\dots,\lvert\F_{4n-1}\rvert$. It follows
that
\[
\frac {(N-4n-1)^{2n}}{m^{4n}}\le\prod_{i=1}^{2n}\lvert\F_{2i-1}\rvert\le t^nN^n,
\]
showing that
\[
t\ge \frac {(N-4n-1)^{2}}{m^{4}N}> \frac {N}{m^4}-1>\frac N{2m^4}.
\]
Finally, let $J$ be the set of subscripts $j$ with $\lvert\F_j\rvert\ge t$.
Clearly, every transversal of the system $\F_j$, ($j\in J$) defines an $n$-cell.
\end{proof}

\begin{remark}
This proof, together with Theorem~\ref{mainresult} shows that one can choose
$C_n$ roughly
$2^{-32n}n^2$ and large enough
$\lvert\F\rvert$ means $\lvert\F\rvert\ge
2m^4=2\ESl^4(4n+1)$ which is about $2^{32n}n^{-2}$
\end{remark}

Here is another positive fraction version of the Erd\H os-Szekeres theorem for
lines:

\begin{theorem}\label{th:posfr2} For every integer $n \ge 2$ there is $C_n'>0$
such that the following holds. Assume $\F$ is a finite family of lines which
is partitioned into nonempty subsets $\F_1,\dots,\F_m$ where
$m=\ESl(n)$. Then there are $n$ subscripts $1\le j_1< j_2 <\dots <j_n\le m$
and subsets $\h_{j_i}\subset \F_{j_i}$ with $\lvert\h_{j_i}\rvert\ge C_n'
\lvert\F_{j_i}\rvert$ (for all $i$) such that every transversal of the
system $\h_{i_1}\dots,\h_{i_n}$ defines
an $n$-cell.
\end{theorem}

For the proof we need the so called same type lemma for lines from
\cite{BP2013}. Three (nonempty) families $\F_1,\F_2,\F_3$ of lines in general
position in the plane are said to have the \emph{same type property} if for every
$\{i,j,k\}=\{1,2,3\}$ and for every choice of $a_i,b_i \in \F_i$, $a_j,b_j \in
\F_j$, no line in $\F_k$ separates the intersection points $a_i\cap a_j$ and
$b_i\cap b_j$. Next, let $\F_1,\dots,\F_m$ ($m\ge 3$) be finite families of
lines with $\bigcup \F_i$ in general position. These families have the
\emph{same type property}
if every three of them have it. Here is the same type
lemma for lines from \cite{BP2013}.

\begin{lemma}\label{lem:stype} For every $m\ge 3$ there is a constant $D_m>0$
such that given finite and nonempty families, $\F_1,\dots,\F_m$, of lines
whose union is in general position, there are subfamilies $\h_i \subset
\F_i$ with $\lvert\h_i\rvert \ge D_m\lvert\F_i\rvert$ (for every $i$) that
have the
same type property.
\end{lemma}

\begin{proof}[Proof of Theorem~\ref{th:posfr2}]
The same type lemma guarantees the existence of subfamilies $\h_i \subset
\F_i$ with $\lvert\h_i\rvert \ge D_m\lvert\F_i\rvert$ (for every $i$) with the
same type property.
Let $m=\ESl(n)=O(4^n/\sqrt n)$ so here
$\lvert\h_i\rvert\ge C_n'\lvert\F_i\rvert$ with $C_n'=D_m$.
Choose a line $a_i$ from every family $\F_i$. By Theorem~\ref{mainresult},
there are subscripts $1\le j_1< j_2 <\dots <j_n\le m$ so that the lines
$a_{j_1},\dots,a_{j_n}$ define an $n$-cell $P$. Suppose first that it is
unbounded, its vertices are
$x_1,\dots,x_{n-1}$ in this order along the boundary of $P$ where
$x_i=a_{j_i}\cap a_{j_{i+1}}$ for $i=1,\dots,n-1$.

We claim that the families $\h_{ j_1},\h_{j_2},\dots,\h_{j_n}$ satisfy the
requirements.
For every $i$, $1\le i\le n$, let
$b_{j_i} \in \h_{j_i}$.
We
want to prove that the set of
these lines also define an $n$-cell. For this it suffices to show
that for a fixed $i$,
replacing $a_{j_i}$ by $b_{j_i}$ in the system $a_{j_1},\dots,a_{j_n}$,
the new system defines an $n$-cell as well. For simpler writing set $\ell_i=a_{j_i}$
and $b=b_{j_i}$.

Assume $n\ge 4$ and let
$x_{i-2},x_{i-1},x_i,x_{i+1}$ be consecutive vertices of $P$. Set
$x=\ell_{i-1}\cap \ell_{i+1}$. By the same type property, $b$ intersects the
segment $[x_{i-2},x]$ as otherwise $\ell_{i-2}$ or $\ell_{i+1}$ would separate
the points $x_{i-1}=\ell_{i-1}\cap \ell_i$ and $\ell_{i-1}\cap b$. Similarly $b$ intersects the segment
$[x,x_{i+1}]$. Consequently $b$ intersects the boundary of the triangle
$x_{i-2}xx_{i+1}$ on the segments $[x_{i-2},x]$ and $[x,x_{i+1}]$. So the new
cell is indeed an $n$-cell.

The same method works if $P$ is a bounded cell, just set $x_n=\ell_n\cap
\ell_1$, and write subscripts modulo $n$. An almost identical approach is to
be used when $P$ is unbounded and $i=1,2,n-1,n$. We omit the straightforward
(but tedious) details.
\end{proof}

\begin{remark}
The constant $D_m$ in the same type lemma is $2^{-16m^2}$ from \cite{BP2013},
giving a doubly exponential (in $n$) bound for $C_n'$. Note further that
Theorem~\ref{th:posfr2} implies Theorem~\ref{th:posfr}, actually a slightly
stronger version, but with a much weaker constant. We mention further that a
very general theorem on semi-algebraic sets due to Fox et al. in
\cite{FGL+2012} can also be used to prove Theorems~\ref{th:posfr} and
\ref{th:posfr2}, but the resulting constants are again much weaker.
\end{remark}

\section{Specified number of intersecting lines}

For any point set $P$ in general position in the plane, let $I(P)$ denote the
interior points of $P$, that is, those points of $P$ which are not on the
boundary of the convex hull of $P$.

The following conjecture of Avis et. al. \cite{AHU2001} is closely related to the Erd\H os-Szekeres Theorem:
For every $n$ there is a $g(n)$ such that any point set $P$ in the plane with
$\lvert I(P)\rvert\ge g(n)$, contains a subset $S\subseteq P$ such that there
are \emph{exactly} $n$ points of $P$ in the convex hull of $S$.
It is known only that $g(n)$ exists
for $n=0, 1, 2, 3$ (see \cite{AHU2001, WD2009}). However, the analogous problem for
lines is very easy.

\begin{theorem}
Let $\mathcal F$ be a finite family of lines. Assume a subfamily $\G\subset\F$
spans a cell $D$, and at least $n$ lines from $\F \setminus \G$ intersect
$D$. Then there is a subset $\h \subset\F$ defining a cell $C$, such that
\emph{exactly} $n$ lines from $\F \setminus \h$ intersect $C$.
\end{theorem}

\begin{proof}
Suppose that $D$ is a $k$-cell defined by the subfamily $\G=\{a_1,\dots,a_k\}$ and $D$ is intersected
by lines $\ell_1, \ell_2, \ldots , \ell_m \in \F \setminus \G$, $m\ge n$. Let
$c$ be a point on the boundary of $D$ which is not on any of the lines $\ell_1, \ell_2,
\ldots , \ell_m$. For $i \in [m]$, let $s_i$ be the segment (or halfline)
$\ell_i \cap D$. Define a partial ordering among $\ell_i$, $i\in [m]$,
as follows. Let
$\ell_i\prec\ell_j$ if any only if $s_i$ and $s_j$ are disjoint and $s_i$
separates $s_j$ from point $c$ in $D$. Now let $\ell_i$ be a \emph{maximal} (minimal would do equally well)
element with respect to $\prec$. Then $\ell_i$ splits the cell $D$ into two
cells; let $C$ be the half containing the point $c$. The lines bounding $C$ form
a subfamily $\h \subset \F$ with $\ell_i \in \h$. Every line $\ell_j, j\ne i$
intersects $C$, but no others.
Therefore, $C$ is intersected by $m-1$ lines. So we can always
decrease the number of crossing lines by one. This proves the theorem.
\end{proof}

\section{Points and lines mixed}

Assume we have a family $\F$ of $N$ lines and a family $S$ of $N$ points in
the plane, in general position. Given a (possibly unbounded) $2n$-gon $K$, we
say that $K$ is determined by $(\F,S)$ if it has $n$ vertices from $S$ and $n$
edges that are subsets of $n$ lines form $\F$. Can one always guarantee the
existence of a $2n$-gon determined by $(\F,S)$?

The answer to this question is no if $n>2$. A simple example showing this can
be obtained if $\F$ defines a cup $P$, and all the points in $S$ are separated
from $P$ by every line in $\F$.

Something can still be done, for example if we allow $S$ to be
translated. Recall $f_l$ from Section~\ref{sec:upper} or from
Lemma~\ref{lem:cupcap}.

\begin{theorem}
Let $N=f_l(2n-1,2n-1)$. If $\F$ is a family of at least $N$ lines and $S$ is a
family of $N$ points, then there is a translation vector $t$ and a convex
$2n$-gon $K$ determined by $(\F,t+S)$.

\end{theorem}

\begin{proof}
Because of the definition of $f_l$, $\F$ spans a $(2n-1)$-cup or
$(2n-1)$-cap. The analogous statement for points also holds for $S$. If $\F$
determines a $(2n-1)$-cup and $S$ determines a $(2n-1)$-cap (or vice-versa)
then we can clearly find a vector $t$ and a $(4n-2)$-gon determined by
$(\F,t+S)$.

Therefore we may assume that $\F$ spans a $(2n-1)$-cup
$\{a_1,\dots,a_{2n-1}\}$ and $S$ determines a $(2n-1)$-cup
$\{p_1,\dots,p_{2n-1}\}$. Here we order the lines by slope and the points by
their $x$-coordinate.

Let $p$ be a point on the line $a_n$ and above all other lines $a_i$. Let $a$
be a line containing $p_n$ and having all the other points $p_i$ above
it. After applying a translation on $S$, we may assume that $p=p_n$. There are
two possibilities for the lines $a$ and $a_n$, since they are symmetric we
only deal with one of them.

If the slope of $a_n$ is smaller than or equal to the slope of $a$, let $q_1$
be a point in $a_1$ above all other lines $a_i$, and $q_i=a_{i-1}\cap a_i$ for
$i=2,\dots,n$. Then the $2n$-gon $K$ with vertices
$q_1,\dots,q_n,p_n,\dots,p_{2n-1}$ is determined by $(\F,S)$.
\end{proof}

\section{Final remarks}

Not surprisingly, the higher dimensional case of the Erd\H os-Szekeres line
theorem follows from the 2-dimensional one. Indeed, suppose that $d>2$ and
we have $ES_l(n)$
hyperplanes in general position in $\R^d$, that is, no two are parallel
and no $d+1$ of them have a common point.
Then
there are always $n$ among
them, say $h_1,\dots,h_n$ and a convex polytope
$P$ such that each $h_i$
contains a $(d-1)$-dimensional face of $P$. The proof is easy: one chooses a
2-dimensional plane $L$ in general position in $\R^d$. The hyperplanes intersect
$L$ in lines in general position that span an $n$-cell $Q$ in $L$, defined
by lines $h_1\cap L,\dots,h_n\cap L$. Then $Q$ is the intersection of $L$ with
the cell $P$ which is a connected component of the complement of the union of
$h_1,\dots,h_n$. Consequently $P$ has the required properties.

What happens if, instead of lines, a finite family $\F$ of halfplanes is given
in the plane? In this case, no matter how large $\lvert\F\rvert$ is, one can't
find $n$ halfplanes whose intersection is a convex $n$-gon (or $n$-cell). The
example showing this comes from the regular convex $N$-gon which is the
intersection of $N$ halfplanes: the family of the complementary halfplanes has
no subset of size $n>4$ whose intersection is an $n$-cell. On the other hand,
given $N=ES_l(2n)$ halfplanes in the plane (with their bounding lines in
general position), there are always $n$ among them so that either their
intersection, or the intersection of their complements is an $n$-cell. This is quite
simple: the $N$ bounding lines define a $2n$-cell $P$ and either half of the
corresponding halfplanes contain $P$ or half of the complementary halfplanes
contains it, and so they define a convex $n$-cell.

\section{Acknowledgements}

Research of the first and second authors was partially supported by ERC
Advanced Research Grant no 267165 (DISCONV), and research of the first and
third authors by Hungarian Science Foundation Grant OTKA K 83767 and K 111827.
Research of the third author was also supported by Hungarian Science
Foundation Grant OTKA NN 102029 under the EuroGIGA programs ComPoSe and GraDR.
We would like to thank Miguel Raggi who offered to write the program that
determined the value of $\ESl(5)$.

\end{document}